\documentclass[a4paper,reqno,11pt]{amsart}
\usepackage{mathrsfs, amsthm, amsmath,amsfonts,amssymb}
\usepackage{mathtools}
\usepackage{commath}
\usepackage{fullpage}
\usepackage{xcolor}
\usepackage{enumerate}
\usepackage{ulem}

\definecolor{eb}{rgb}{0.2,0,0.8}

\usepackage{amsbsy}
\usepackage{amsopn}
\usepackage[english]{babel}
\usepackage[T1]{fontenc}

\theoremstyle{plain}
\newtheorem{theorem}{Theorem}[section]
\newtheorem{lemma}[theorem]{Lemma}
\newtheorem{proposition}[theorem]{Proposition}
\newtheorem{corollary}[theorem]{Corollary}

\theoremstyle{definition}
\newtheorem{definition}[theorem]{Definition}
\newtheorem{remark}[theorem]{Remark}

\usepackage{hyperref}

\numberwithin{equation}{section}

\DeclareMathOperator{\er}{\mathbb{E}}

\DeclareMathOperator{\p}{\mathbb{P}}
\DeclareMathOperator{\e}{\mathbb{E}}
\DeclareMathOperator{\ee}{\mathbf{e}}

\DeclareMathOperator{\nr}{\mathbb{N}}

\DeclareMathOperator{\re}{\mathbb{R}}

\usepackage{mathpazo}
\usepackage[a4paper, scale={0.72,0.74}, marginratio={1:1}, footskip=7mm, headsep=10mm]{geometry}

\usepackage{indentfirst}

\makeatletter
\renewcommand{\@secnumfont}{\bfseries}
\renewcommand\section{\@startsection{section}{1}%
\z@{.7\linespacing\@plus\linespacing}{.5\linespacing}%
{\large\bfseries\scshape\centering}}
\renewcommand\subsection{\@startsection{subsection}{2}%
  \z@{.5\linespacing\@plus.7\linespacing}{-.5em}%
  {\bfseries\scshape}}
\renewcommand\subsubsection{\@startsection{subsubsection}{3}%
  \z@{.5\linespacing\@plus.7\linespacing}{-.5em}%
  {\scshape}}
\makeatother

\author[J.~Arista]{Jonas Arista}
\address{Universit\"at Bielefeld\\
Fakult\"at f\"ur Mathematik \\
Universit\"atsstra{\ss}e 25\\
33615 Bielefeld, Germany}
\email{jarista@math.uni-bielefeld.de}

\author[V.~Rivero]{V\'ictor Rivero}
\address{Centro de Investigaci\'on en Matem\'aticas A.C. \\
Calle Jalisco s/n \\
Col. Mineral de Valenciana \\
C.P. 36240 Guanajuato, Mexico}
\email{rivero@cimat.mx}



\begin{document}

\title{Implicit renewal theory for exponential functionals \\ of L\'evy processes} 
\maketitle

\begin{abstract}
We establish a new integral equation for the probability density of the exponential functional of a L{\'e}vy process and provide a three-term (Wiener-Hopf type) factorisation of its law. We explain how these results complement the techniques used in the study of exponential functionals and, in some cases, provide quick proofs of known results and derive new ones. We explain how the factors appearing in the three-term factorisation determine the local and asymptotic behaviour of the law of the exponential functional. We describe the behaviour of the tail distribution at infinity and of the distribution at zero under some mild assumptions.
\end{abstract}


\section{Introduction}

Let $\xi=(\xi_t , t\ge0)$ be a real valued L\'evy process killed at an independent exponential time $\zeta$ of parameter $q\ge 0$ and $\xi_t =-\infty$ for all $t\ge\zeta$. When $q=0$, this states that $\zeta=\infty$ and that $\lim_{t\to\infty}\xi_{t}=-\infty$ almost surely.

According to Theorem 1 in \cite{BY}, the above are the only cases under which the {\it exponential functional} associated to $\xi$:

\begin{equation*}
I_\xi:=\int^{\zeta}_{0}ds\,e^{\xi_{s}}=\int^{\infty}_{0}ds\,e^{\xi_{s}},
\end{equation*}
is finite almost surely. Exponential functionals are ubiquitous objects in probability theory and stochastic modelling, and their study was initiated by Urbanik~\cite{urbanik}. Because of its importance, the obtainment of distributional properties and tools for characterising its law or asymptotic behaviour has been at the source of many researches in the last two decades, see e.g.~\cite{CPY, yor-yellow-book, BY, patie-pardo-savov, patielaw, patie-savov1, patie-savov-cras, MR3370671}  and, more recently \cite{patie+savov+1,BLR,haas,Min-Sav}. Among the latter references,  we underline \cite{patie+savov+1} since it surveys a large part of the existing literature and, furthermore, develops powerful analytical methods, based on complex analysis, to obtain sharp distributional properties of $I_\xi$.

In this article, our contribution to the very active area of research on exponential functionals is twofold. First, we prove a new integral equation for the probability density function of $I_\xi$ that is written in terms of the so-called {\it renewal} (or {\it potential}) {\it measure} $U_\xi$ of $\xi$, see Theorem~\ref{keythm}. This opens the gate to the use of renewal theoretic arguments for the study of the distribution of $I_\xi$. And second, we provide a three-term (Wiener-Hopf type) factorisation of $I_\xi$, see Theorem~\ref{PPS2}, explaining how the factors involved determine the asymptotic and local behaviour of its distribution. We apply our results to describe the behaviour of the tail distribution at infinity of $I_{\xi}$ and of the distribution at zero under some mild assumptions, see Section~\ref{EstimatesdistributionI}. Our initial motivation for pursuing this work was to find useful alternative tools when expressions for the {\it characteristic exponent} of $\xi$ are not explicit.

 Below, we further discuss our main results introduced in the previous paragraph. We refer to Section~\ref{Notation} for definitions.

In \cite{bertoin+lindner+maller} it was proved that the probability distribution of $I_\xi$ admits a density function. Under some assumptions, this density function satisfies integral and integro-differential equations given in terms of the {\it characteristics} of the L\'evy process $\xi$ (see e.g.~\cite{CPY,Kuznetsov-pardo-savov, Pardo+rivero+Vanschaik,BLR}). Equations of this type are very helpful in obtaining several distributional properties. Drawing a parallel between the latter references and our integral equation presented in~Theorem~\ref{keythm} in Section~\ref{Sec:IntegralEquation}, one can think of those equations as given in terms of the infinitesimal generator of $\xi$, whilst ours is given in terms of the potential operator of $\xi$. One would hence expect that it is possible to pass from one to the other and, in some sense, this is the case, see e.g. Corollary~\ref{int-equation} where we establish an equation very closely related to the one first appeared in \cite{CPY}. However, the equation in Corollary~\ref{int-equation} is given in terms of the characteristics of the {\it upward} and {\it downward ladder height processes} to emphasise the role played by the Wiener-Hopf factorisation of $\xi$.

The three-term factorisation of $I_\xi$, presented in~Theorem~\ref{PPS2} in Section~\ref{threeterm}, has the form
\begin{equation}\label{3term}
I_\xi\stackrel{Law}{=}e^{\overline{\xi}_{\infty}}\frac{I_{-\widehat{H}}}{R_{H}},
\end{equation} 
where the three terms on the right-hand side are independent random variables. The term $\overline{\xi}_{\infty}$ denotes the supremum of $\xi$, $R_H$ is a variable whose law depends on the  upward ladder height process of $\xi$, whilst $I_{-\widehat{H}}$ is a copy of the exponential functional associated to the corresponding  downward ladder height process. Notice that the law of the former two variables depends on the behaviour of the supremum, whilst that of the latter depends on the behaviour of the infimum.  
Equation~\eqref{3term} complements the Wiener-Hopf type factorisation for exponential functionals obtained by Pardo, Patie and Savov \cite{patie-pardo-savov} and later generalised by Patie and Savov \cite{patie-savov1, patie-savov-cras}.

In Section~\ref{EstimatesdistributionI}, we explain in detail how our main results provide an alternative to some of the techniques commonly used in the study of exponential functionals and hence provide quick proofs of some known results, obtain new ones (notably Theorem~\ref{th:asym1} and Theorem~\ref{asym:2}). Furthermore, we unravel the dependence of the factors in the three-term factorisation~\eqref{3term} on the asymptotic and local behaviour of the distribution of $I_\xi$. This adds up to the available tool kit to study exponential functionals of L\'evy processes. 

Let us mention that the possibility of applying renewal theory to the study of exponential functionals was indirectly known by Goldie~\cite{goldie}, where, for {\it perpetuities}, renewal sequences are built to study their tail distributions. The exponential functional $I_\xi$ is a perpetuity, since, by the strong Markov property,
\begin{equation}\label{perpetuity}
I_\xi\stackrel{Law}{=}Q+M I,
\end{equation}
for certain $(Q,M)$ independent of $I_\xi$. The advantages of our approach over Goldie's are two. First, the connection with renewal measures through the integral equation in Theorem~\ref{keythm} is transparent and, second, the factorisations of $I_\xi$ in Theorem~\ref{PPS} and Theorem~\ref{PPS2} allow us to describe both the behaviour of the tail distribution at infinity (Theorem~\ref{th:asym1} and Theorem~\ref{VRrtb}) and that of the distribution at zero (Theorem~\ref{asym:2} and Theorem~\ref{corollarysmtg}). Furthermore, we use the integral equation in~Theorem~\ref{keythm} and Stone's decomposition theorem (Appendix~\ref{App:Stone}) to quantify the rate of convergence of the estimates in Section~\ref{rateconvergence}. These results are not as sharp as those obtained in \cite{patie-savov1, haas, Min-Sav} but it makes the point of finding estimates in terms of the L\'evy measure of $\xi$ instead of the characteristic exponent.

Regarding the natural question of determining whether there is a unique solution to the integral equation in Theorem~\ref{keythm}, we conjecture that the answer is positive in general. A first version of the present article partially answered this question but we have omitted the details here for the purpose of the briefness of the exposition.

\vspace{3mm}
\noindent
\textbf{Organisation of the article.}
In Section~\ref{Notation} we introduce some notation and state versions of the Wiener-Hopf factorisation for a L\'evy process. The main results of the paper are contained in Section \ref{mainsec}. In particular, Theorem~\ref{keythm} establishes an integral equation for the density of the exponential functional $I_\xi$ and, in Theorem~\ref{PPS2}, we provide a three-term factorisation of $I_\xi$. Some consequences of our results regarding the behaviour at infinity, then at zero, of the distribution of the exponential functional are collected in Section~\ref{EstimatesdistributionI}. In Section \ref{theothereq} we obtain further distributional identities for $I_\xi$. Finally, proofs of the main results are presented in Section~\ref{proofs}.

\vspace{3mm}
\noindent
\textbf{Acknowledgements.}
The second author acknowledges support from EPSRC grant number EP/M001784/1 and from CONACyT grant 234644. This work was undertaken whilst the second author was on sabbatical at the University of Bath, he gratefully acknowledges the kind hospitality of the department and university.

\subsection{Notation and preliminaries}
\label{Notation}

Here we introduce some notation and preliminary notions that we use throughout this work. Most of the results presented in this section are known (see e.g. \cite{bertoin-book, doney-book, kyprianou-book} and references therein), except for Lemma~\ref{lemma:fR} which plays a crucial role in Section~\ref{threeterm}.

 As mentioned above, $\xi=(\xi_t , t\ge0)$ will denote a real valued L\'evy process killed at an independent exponential random variable $\zeta$ of parameter $q\ge0$ and $\xi_t =-\infty$ for all $t\ge\zeta$. When $q=0$, then $\zeta=\infty$ and we assume $\lim_{t\to\infty}\xi_{t}=-\infty$ a.s.

The process $\xi$ is characterised by a quadruplet $(q, a, Q, \Pi)$ where $q\geq 0$ is the killing rate, $a\in\re$ is the linear term, $Q\ge0$ is the Gaussian term and $\Pi$ is the L\'evy measure. The quadruplet $(q, a, Q, \Pi)$ is also known as the {\it characteristics} of the process $\xi$.  The {\it characteristic exponent} of $\xi$ is denoted by $\Psi:\re\to\mathbb{C},$ and it is given by
$$\Psi(\lambda):=q+ai\lambda + \frac{Q^{2}\lambda^{2}}{2}+\int_{\re\setminus\{0\}}\Pi(dx) (1-e^{i\lambda x}-i\lambda x1_{\{|x|<1\}}),\qquad \lambda \in \re,$$
satisfying
$$\er\left(e^{i\lambda \xi_{1}}, 1<\zeta \right)=\exp\{-t\Psi(\lambda)\}.$$

The {\it renewal} (or {\it potential}) {\it measure} of $\xi$ is the $\sigma$-finite measure on $\mathbb{R}$ defined by
\begin{equation}\label{eq:renewalmeasure}
U_\xi(dy):=\mathbb{E}\left(\int^{\zeta}_{0} dt 1_{\{\xi_{t}\in dy\}}\right)=\mathbb{E}\left(\int^{\infty}_{0} dt 1_{\{\xi_{t}\in dy\}}\right),\qquad y\in \re.
\end{equation}
Naturally, we may see $U_\xi$ as an integral operator acting on suitable measurable functions $f:\re\to[0,\infty)$ as
\begin{equation}\label{eq:potentialmeasure}
\int_{\re}U_\xi(dy)f(y)=\int_{0}^{\infty}dt\e\left(f(\xi_t)\right).
\end{equation}
The name `renewal measure' is justified by the fact that $U_\xi(dy)$ can be written as
\begin{equation}\label{renewalU}
\begin{split}
U_\xi(dy)=\sum_{n\geq 1}F^{*n}(dy),
\end{split}
\end{equation} 
where $F(dy):=\p(\xi_{\ee}\in dy)$  and $\ee$ is an exponential random variable of parameter $1$ independent of $\xi$.

\subsubsection{Wiener-Hopf factorisation}
\label{subsection11}

We denote by $H=(H_{t}, t\geq 0)$, resp. $\widehat{H}=(\widehat{H}_{t}, t\geq 0)$, the upward, resp. downward, ladder height process of $\xi$. Recall that $H$ and $\widehat{H}$ are {\it subordinators}, that is non-decreasing L\'evy processes. We denote by $(\kappa(q,0), b, \Pi_{H})$, resp. $(\widehat{\kappa}(q,0), \widehat{b}, \Pi_{\widehat{H}}),$ the killing term, linear term and L\'evy measure of $H$, resp. $\widehat{H}$. The Laplace exponent of $H,$ resp. $\widehat{H},$ will be denoted by $\kappa(q,\cdot)$, resp. $\widehat{\kappa}(q,\cdot)$, and it can be expressed as 
$$\kappa(q,\lambda)=-\log\e\left(e^{-\lambda H_{1}}\right)=\kappa(q,0)+b\lambda+\int_{(0,\infty)}\Pi_{H}(dx) (1-e^{-\lambda x}),\ \lambda\geq 0,$$ 
and a similar formula holds for $\widehat{\kappa}(q,\cdot).$  Since we are assuming that either $\xi$ has a finite lifetime or drifts towards $-\infty$, ${H}$ has a finite lifetime, equiv. $\kappa(q,0)>0,$ and $\widehat{H}$ has a finite lifetime if and only if $\xi$ has a finite lifetime.  

The tail L\'evy measure of the upward ladder height subordinator $H$ is defined by $\overline{\Pi}_{H}(x):=\Pi_{H}(x,\infty)$, $x>0$, and a similar formula holds for $\widehat{H}$.

The celebrated {\it Wiener-Hopf factorisation}, in its Fourier transform representation, establishes that, for $q>0$
\begin{equation}
\frac{q}{\Psi(\lambda)}=\frac{\kappa(q,0)}{\kappa(q,-i\lambda)}\frac{\widehat{\kappa}(q,0)}{\widehat{\kappa}(q,i\lambda)}, \qquad q=c\kappa(q,0)\widehat{\kappa}(q,0),\qquad \lambda \in \re,
\end{equation}for a constant $c>0$; while, for $q=0$
\begin{equation}
\Psi(\lambda)=c\kappa(0,-i\lambda)\widehat{\kappa}(q,i\lambda), \qquad \lambda \in \re.
\end{equation}for a constant $c>0$.

The above factorisation is at the root of the development of the fluctuation theory of L\'evy processes and can be expressed in many other equivalent forms. One of them, obtained by Fourier inversion, ensures that the renewal measure~\eqref{eq:renewalmeasure} can be decomposed in terms of the renewal measures of $H$ and $\widehat{H}$, given by
$$V_{H}(dy)=\er\left(\int^{\infty}_{0}dt1_{\{H_{t}\in dy\}}\right),\qquad V_{\widehat{H}}(dy)=\er\left(\int^{\infty}_{0}dt1_{\{\widehat{H}_{t}\in dy\}}\right),\ y\geq 0,$$
by means of the formula
\begin{equation}\label{WHpot}
\int_{\re}U_\xi(dy)f(y)=K\int_{[0,\infty)}\int_{[0,\infty)}V_{H}(du)V_{\widehat{H}}(dv)f(u-v), 
\end{equation}
for every $f:\re\to \re$ measurable and positive, where $0<K<\infty$ is a fixed constant whose value depends only on the normalisation of the local time at the supremum and infimum of $\xi$, and we can assume without loss of generality that $K=1$. Another version of the Wiener-Hopf factorisation, obtained by Vigon~\cite{vigon},  establishes a relation between the measures $\Pi,$ $\Pi_{H}$ and $V_{\widehat{H}},$ see the Section 5.3 in \cite{doney-book} or Theorem 6.22 in \cite{kyprianou-book} for a precise statement. 

Notice that, since $H$ has finite lifetime,  the measure $V_{H}$ is a finite measure and the law of the overall supremum of $\xi$, namely $\overline{\xi}_{\infty}:=\sup_{0\leq s<\zeta}\xi_{s}$, is given by
\begin{equation}
\label{supremum}
\p(\overline{\xi}_{\infty}\in dy)=\kappa(q,0) V_{H}(dy),\qquad y\geq 0.
\end{equation}
This equation is at the root of the three-term factorisation of the exponential functional $I_{\xi}$ presented in Theorem~\ref{PPS2}. Moreover, if $\xi$ has a finite lifetime (equiv. $q>0$) then $\widehat{\kappa}(q,0)>0$ and, by duality, the law of $\underline{\xi}_{\infty}:=-\inf_{0\leq s<\zeta}\xi_{s}$  is given by
\begin{equation}
\label{infimum}
\p(\underline{\xi}_{\infty}\in dy)=\widehat{\kappa}(q,0) V_{\widehat{H}}(dy),\qquad y\geq 0.
\end{equation}

\subsubsection{Residual exponential functionals}\label{residualexponential}
In what follows, we will require the following key result. Hereafter and unless otherwise stated, for an $n$-tuple of non-negative random variables $(X_1,X_2,\dots,X_n)$ in a product $\prod_{i=1}^{n}X_{i}$, we will assume that there is independence of the factors.

\begin{lemma}[Bertoin and Yor \cite{BY2001, BY}]\label{lemmaBY}
Let $\xi=-\sigma,$ with $\sigma$ a possibly killed subordinator with Laplace exponent defined by
$$\phi(\lambda):=-\log\e\left(e^{-\lambda\sigma_{1}}, 1<\zeta\right),\qquad \lambda \geq 0.$$
Consider the corresponding exponential functional $I_{-\sigma}$. Then, there exists a (non-negative) random variable $R_{\sigma}$ whose law is determined by its entire moments, which satisfy the recurrence relation 
\begin{equation}\label{recurrencerelation}
\e(R^{\lambda}_{\sigma})=\phi(\lambda)\e(R^{\lambda-1}_{\sigma}), \qquad \lambda>0.
\end{equation} 
The identity is true in the limit sense if $\lambda=0.$ In particular,
$$\e(R^{n}_{\sigma})=\prod^{n}_{k=1}\phi(k),\qquad n\in \nr.$$ 
Moreover,
$$R_{\sigma}I_{-\sigma}\stackrel{Law}{=}\ee_{1},$$ with $\ee_{1}$ an exponential random variable of parameter $1$. 
\end{lemma} 

The random variable $R_{\sigma}$ in Lemma~\ref{lemmaBY} is called the {\it residual exponential functional} associated to the subordinator $\sigma$. Hirsch and Yor~\cite{Hirsch-Yor} translated the recurrence relation~\eqref{recurrencerelation} into a functional equation for the density of $R_{\sigma}$, when it exists, under the assumption that the renewal measure $V_{\sigma}$ of $\sigma$, has a density on $(0,\infty)$. Namely, if $v_{\sigma}$ denotes the density of $V_{\sigma}$, they proved that $R_{\sigma}$ has a density $f_{R_{\sigma}}$ and it satisfies the equation 
$$f_{R_{\sigma}}(t)=\int^{\infty}_{t}dy\,v_{\sigma}(\log(y/t)) f_{R_{\sigma}}(y) =\int^{\infty}_{0}V_{\sigma}(dz) f_{R_{\sigma}}(te^{z})te^{z}.$$

If we remove the assumption that $V_{\sigma}$ has a density on $(0,\infty)$, the identity above is preserved but in measure form:

\begin{lemma}
\label{lemma:fR}
Using the notation of Lemma~\ref{lemmaBY},   we have the equality of measures
$$\frac{1}{t}\p(R_{\sigma}\in dt)=\int_{[0,\infty)}V_{\sigma}(dy)\p(e^{-y}R_{\sigma}\in dt),\qquad \text{on} \ (0,\infty).$$ 
\end{lemma} 

We provide a probabilistic proof of Lemma \ref{lemma:fR} in Section \ref{proofs}.

\section{Main results}
\label{mainsec}

In this section we state the main results of this article. We use the terminology and notation introduced in Section~\ref{Notation}.  Recall that the exponential functional associated with the L\'evy process $\xi$ is defined by the random variable
\begin{equation*}
I_\xi=\int^{\zeta}_{0}ds\,e^{\xi_{s}} =\int^{\infty}_{0}ds\,e^{\xi_{s}},
\end{equation*} 
and that its probability distribution admits a density function (see e.g. \cite{bertoin+lindner+maller} and \cite{Pardo+rivero+Vanschaik}).

In the following, we will write $I=I_\xi$ when the dependence on the underlying process $\xi$ is clear. Analogously, we will write $U=U_\xi$ for the renewal measure.

\subsection{Integral equation}
\label{Sec:IntegralEquation}

Here, we establish a new integral equation for the probability density function of the exponential functional $I$ that is given in terms of the renewal measure of $\xi$. As commented in the introduction, this will allow us to use renewal theoretic arguments for the study of the distribution of $I$, see e.g. Sections~\ref{rateconvergence} and~\ref{theothereq}, and Remark~\ref{remarkrenewaltail}.

\begin{theorem}\label{keythm}
Let $U(dy)$ be the renewal measure of $\xi$ and denote by $k$ the probability density function of the exponential functional $I$. Then
\begin{equation}\label{pot-eq}
\int^{\infty}_{t}ds\,k(s)=\int_{\re}U(dy) k(te^{-y}),\qquad  \text{a.e. $t$ on } (0,\infty).
\end{equation}
\end{theorem}

\begin{remark}
When $\xi$ is the negative of a subordinator, the validity of (\ref{pot-eq}) was established in \cite{Pardo+rivero+Vanschaik} using an identification of the moments of $I$. The proof of Theorem~\ref{keythm} provided in Section \ref{proofs} uses a different and elementary argument.
\end{remark}

In terms of the Wiener-Hopf factorisation of the renewal measure $U$ given in~\eqref{WHpot}, the integral equation~(\ref{pot-eq}) can be expressed, using~\eqref{supremum}, as
\begin{equation*}
\p(I>t)=\int_{[0,\infty)}\p(\overline{\xi}_{\infty}\in dy)\mathcal{L}(te^{-y}), \qquad t>0,
\end{equation*}
with $$ \mathcal{L}(s):=\frac{1}{\kappa(q,0)}\int_{[0,\infty)}V_{\widehat{H}}(dz)k(se^{z}), \qquad s>0.$$
Analogously, using~\eqref{infimum}, if $\xi$ has a finite lifetime then the integral equation~\eqref{pot-eq} reads as
\begin{equation*}
\kappa(q,0)\widehat{\kappa}(q,0)\p(I>t)=\er\left(k(t\exp\{\underline{\xi}_{\infty}-\overline{\xi}_{\infty}\})\right),\qquad t\geq 0,
\end{equation*}
where $\underline{\xi}_{\infty}$ and $\overline{\xi}_{\infty}$ are taken as independent copies of the overall supremum and infimum, respectively.

For $z\in\mathbb{C}$, let $\Re z$ denote the real part of $z$. A first consequence of Theorem~\ref{keythm} is the following. 

\begin{corollary}\label{corollary...1}
Let $C$ be the set defined by $$C:=\left\{\lambda\in\mathbb{C}: \Re\lambda> 0\ \text{and}\ |\e[e^{\lambda\xi_{1}}]|<1\right\},$$ and denote still by $\Psi$ the analytic extension of the characteristic exponent of $\xi$ on $C$. The following identity holds true for all $\beta\in C,$  
\begin{equation}\label{recurrence}
\er(I^{\beta})=\frac{\beta}{\Psi(-i\beta)}\er(I^{\beta-1}).
\end{equation}
\end{corollary}

Some versions of Corollary \ref{corollary...1} have been obtained for instance in \cite{CPY, Maulik-Zwart, rivero-convequiv, Kuznetsov-pardo-savov,  patie+savov+book}. In particular, in \cite{patie+savov+book} the authors used a similar identity to identify the Mellin transform of $I$, in which most of their developments are rooted. We are aware that the identity (\ref{keythm}), or some of its previous versions, can be deduced from the identity (\ref{recurrence}) using Fourier inversion methods, nevertheless we prefer to establish (\ref{keythm}) using the probabilistic approach coming from Theorem~\ref{keythm}. This happens for other results, as for instance Lemma~\ref{lemma:fR}, where we have also chosen to give a probabilistic proof.

\subsection{Three-term factorisation}
\label{threeterm}

In this section, we state a three-term factorisation (Theorem~\ref{PPS2}) of the exponential functional $I$ which complements the two-term factorisation (Theorem~\ref{PPS}) obtained by Pardo, Patie and Savov \cite{patie-pardo-savov} and later generalised by Patie and Savov \cite{patie-savov1,patie-savov-cras}.  The proof of Theorem~\ref{PPS2} is presented in Section~\ref{proofs}. The factorisations~\eqref{fact-pps} and~\eqref{equationppsext}  will be essential in deriving the asymptotic and local behaviour of the distribution of $I$ in the forthcoming Section~\ref{EstimatesdistributionI}.

\begin{theorem}[Pardo, Patie and Savov \cite{patie-pardo-savov}; Patie and Savov \cite{patie-savov1,patie-savov-cras}]\label{PPS}
Let $R_{H}$ be the residual exponential functional associated to the upward ladder height subordinator ${H}$ and $I_{-\widehat{H}}$ be the exponential functional of the negative of the downward ladder height $\widehat{H}.$ Assume $I_{-\widehat{H}}$ and $R_{H}$ are independent. We have $\e(R^{-1}_{H})=(\kappa(q,0))^{-1}$ and the random variable $J_{H}$, whose law is defined by
$$\p(J_{H}\in dy)={\kappa(q,0)y}\p\left(\frac{1}{R_{H}}\in dy\right),\qquad y\geq 0,$$ is such that 
\begin{equation}
\label{fact-pps}
I\stackrel{Law}{=}J_{H} I_{-\widehat{H}}.
\end{equation}
\end{theorem}

\begin{theorem}\label{PPS2}
With the notation of Theorem \ref{PPS}, we have the identities 
 \begin{equation}
 \label{equationppsext}
 J_{H}\stackrel{Law}{=}e^{\overline{\xi}_{\infty}} \frac{1}{R_{H}},\qquad I\stackrel{Law}{=}e^{\overline{\xi}_{\infty}} \frac{I_{-\widehat{H}}}{R_{H}},
 \end{equation}
where $\overline{\xi}_{\infty}=\sup_{0\leq s<\zeta}\xi_{s}$. 
\end{theorem}

Theorems \ref{PPS} and \ref{PPS2} have a large potential of applicability, as they give the possibility of deriving properties of exponential functionals using exponential and residual  exponential functionals of subordinators, which are in general simpler objects (see e.g. Proposition~\ref{corollary....2}). In \cite{patie-savov-cras}, the factorisation~\eqref{fact-pps} has been used to provide a useful formula for the Mellin transform of $I$ in terms of generalised Gamma functions, and, in \cite{patie+savov+book}, to obtain spectral decompositions of the semigroup of self-similar Markov and related processes.

 Using Lemma~\ref{lemmaBY} and Theorems~\ref{keythm} and~\ref{PPS}, we prove in Section~\ref{proofs} the following proposition regarding finiteness of some negative moments of $I$. In Theorem 2.18 in \cite{patie+savov+1} further information about the finiteness of the moments can be found.

\begin{proposition}\label{corollary....2}
One has
$$\e(I^{-\delta})<\infty,\qquad \forall \delta\in(0,1).$$ Furthermore, $\e(I^{-1})<\infty$ if and only if $q=0$ and $\e(\xi_{1})<\infty$. In this case, $\e(I^{-1})=\e(\xi_{1})=\kappa(0,0)\e(\widehat{H}_{1})$.
\end{proposition}

Proposition~\ref{corollary....2} can be useful in applying the recurrence of Corollary~\ref{corollary...1}. Indeed, it follows from Proposition~\ref{corollary....2}  and~\eqref{recurrence} that, for $\delta>0,$ $\e(I^{\delta})<\infty$ if and only if $\e(e^{\delta\xi_{1}}, 1<\zeta)<1$ (this fact was first proved in \cite{rivero-convequiv}). Notice also that the finiteness of the moments of negative (resp. positive) order of $I$ depends only on those of $I_{-\widehat{H}}$ (resp.  of $R_{H}$).

\subsection{Estimates for the distribution of $I$}
\label{EstimatesdistributionI}

In this section we state some results with respect to the behaviour at infinity, then at zero, of the distribution of the exponential functional $I$ that can be deduced from the main results presented in Section \ref{threeterm}. 

 The results in Sections~\ref{beinfinity} and~\ref{bezero} do not seem to imply similar information for the density function of $I$ or its derivatives. Nevertheless, other results have appeared in recent years studying this problem, see e.g. \cite{patie+savov+1, haas, Min-Sav}. We stress that our approach complements that developed in  the latter references  in the sense that it explains how the factors in the three-term factorisation~\eqref{equationppsext} determine the asymptotic and local behaviour of the distribution of $I$. Moreover, Theorems~\ref{VRrtb} and~\ref{corollarysmtg} provide estimates in terms of the L\'evy measure of $\xi$ instead of its characteristic exponent, which can be a useful alternative when expressions for the characteristic exponent are not explicit.

\subsubsection{The behaviour at infinity}\label{beinfinity}
Throughout this subsection we will assume that $\xi$ is non-monotone, that is, it is not the negative of a subordinator (notice that this does not exclude the case where $\xi$ is a killed subordinator, because such a process has no monotone paths as it eventually jumps to its cemetery state $-\infty$).  We exclude the monotone case from our present discussion since this case has been carefully analysed in \cite{haas} using an integral equation for the density of $I$ given in terms of the characteristics of the subordinator, and in \cite{Min-Sav} using the Mellin transform.

In \cite{rivero-convequiv}, it has been conjectured that the tail behaviour of the distribution of $I$ should satisfy
\begin{equation}\label{conjecture}
\mathbb{P}(\log I>t)\sim c\mathbb{P}(\sup_{s>0}\xi_{s}>t),\qquad t\to\infty,
\end{equation} for some constant $c\in(0,\infty).$
This in turn is based on the heuristic that the large values of $I$ are due to the large values of $e^{\overline{\xi}_{\infty}}.$ The second identity in law in Theorem~\ref{PPS2}
establishes that $$I\stackrel{Law}{=}e^{\overline{\xi}_{\infty}} \frac{I_{-\widehat{H}}}{R_{H}},$$
which implies that there is an even stronger relation between the large values of $I$ and those of the absolute supremum of $\xi.$  To study the asymptotic behaviour of the tail distribution, the latter identity sets us in the context of determining when for independent random variables $X$ and $Y$ we have
$$\p(XY>t)\sim \mathcal{C}\p(X>t),\qquad t\to\infty,$$
 which, in a general setting, has no solution because of the wide range of distributions that are involved. Nevertheless, in the context of heavy tails, and motivated by Breiman's ~\cite{breiman} key result, very important results have emerged that allow to give a precise answer when either tail distribution is regularly varying (see e.g. \cite{jacobsenetal}). For our purposes, those results will provide the necessary tools to establish the conjecture (\ref{conjecture}) in Theorem~\ref{th:asym1}.  Since $I$ is a perpetuity, see~\eqref{perpetuity}, it follows from Theorem 4.1 in \cite{goldieandgrubel} that the law of $I$ has at least a power law (Pareto) tail
$$\liminf_{t\to\infty}\frac{\log\p(I>t)}{\log t}>-\infty.$$ A sharper result is proved in Theorem 2.11 item 1, in \cite{patie+savov+1}. With this information at hand, the following theorem establishes the described conjecture in the most general and tractable case among those possible for exponential functionals of L\'evy processes: that of regular variation. Notice that this theorem goes beyond the Cramer's case.
  
\begin{theorem}\label{th:asym1}
For $\alpha\geq 0,$ the following are equivalent:
\begin{enumerate}
\item[(i)] the function $t\mapsto \p(I>t)$ is regularly varying at infinity with index $-\alpha,$
\item[(ii)] the function $t\mapsto \p(e^{\overline{\xi}_{\infty}}>t)$ is regularly varying at infinity with index $-\alpha.$
\end{enumerate} 
In this case, 
$$\p(I>t)\sim \e\left(I^{\alpha}_{-\widehat{H}}R^{-\alpha}_{H}\right)\p\left(e^{\overline{\xi}_{\infty}}>t\right),\qquad t\to\infty.$$
\end{theorem} 

Conditions under which the statements of Theorem~\ref{th:asym1} hold have been obtained in \cite{Maulik-Zwart,rivero1,rivero-convequiv}. For the reader's convenience, and because we offer some refinements, we include them below. Theorem~2.11 in \cite{patie+savov+1} also gives conditions where the regular variation holds, and furthermore, they provide estimates for the density of $I$ in that setting.
 
\begin{definition}\label{def1}
Let $G$ be a distribution function such that $G(x)<1$ for all $x\in\re$. Write $\overline{G}(x):=1-G(x)$ for $x\in \re$. We say that $G$ is \textit{convolution equivalent} if
\begin{enumerate}
\item[(a)] it has an exponential tail with rate $\gamma\geq 0,$ written $G\in\mathcal{L}_{\gamma},$ viz. 
$$\lim_{x\to\infty}\frac{\overline{G}(x-y)}{\overline{G}(x)}=e^{\gamma y},\qquad y\in\re;$$
\item[(b)] the following limit exists
$$\lim_{x\to\infty}\frac{\overline{G*G}(x)}{\overline{G}(x)}:=2M<\infty,$$ where $*$ denotes the usual convolution.
\end{enumerate}
In this case, we use the notation $G\in \mathcal{S}_{\gamma}$ and the constant $M$ is determined by $M=M_{G}:=\int_{\mathbb{R}}e^{\gamma x}dG(x)$. When $\gamma=0,$ the family $\mathcal{S}_{0}$ is better known as the class of subexponential distributions. When $\gamma>0$, the convergence in (a) above holds uniformly over intervals of the form $(b,\infty),$ for $b\in \re$.
\end{definition}

A result by Pakes \cite{pakes2004, pakes2007} (see also \cite{watanabe}) establishes that if $F$ is an infinitely divisible distribution with L\'evy measure $\nu$ and $$\mu_{\nu}(x):=\frac{\nu[x\vee 1,\infty)}{\nu[1,\infty)},\qquad x\in\re,$$ then we have the equivalences: 
$$\mu_{\nu}\in\mathcal{S}_{\gamma}\ \Longleftrightarrow\ \mu_{\nu}\in\mathcal{L}_{\gamma}\ \text{and} \ \lim_{x\to\infty}\frac{\overline{F}(x)}{\nu[x,\infty)}=M_{F}\ \Longleftrightarrow\ F\in\mathcal{S}_{\gamma},$$ where by $\mu_{\nu}\in\mathcal{S}_{\gamma}$ we mean the distribution in $\mathcal{S}_{\gamma}$ whose right tail equals $\mu_{\nu}.$  We will use the notation $\nu\in\mathcal{S}_{\gamma}$ whenever $\mu_{\nu}\in\mathcal{S}_{\gamma}$. For further background on convolution equivalent distributions we refer to \cite{pakes2004, pakes2007, watanabe} and the references therein. 

For any $\gamma\geq 0$, we will write $\xi\in\mathcal{S}_{\gamma}$ if the law of $\xi_{1},$ conditionally on $\{1<\zeta\},$ is in $\mathcal{S}_{\gamma}$ or equivalently its L\'evy measure $\Pi$ is in $\mathcal{S}_{\gamma}$.

\begin{theorem}\label{VRrtb}
Define $\overline{\Pi}^{+}(x):=\Pi(x,\infty)$, $x>0$. The following holds:

\begin{enumerate}
\item[(i)] Assume $\xi_{1}$ is non-lattice and that there exists $\theta>0$ such that
$$
\mathbb{E}\left( e^{\theta \xi_1},1<\zeta\right)=1, \quad \text{ and }\quad \mathbb{E}\left(\xi_1^+ e^{\theta \xi_1},1<\zeta \right)<\infty.
$$ 
Then,
$$t^{\theta}\mathbb{P}(I>t)\xrightarrow[t\to\infty]{} \frac{\e(I^{\theta-1})}{\e(\xi_{1}e^{\theta\xi_{1}})}\in(0,\infty).$$ 
\item[(ii)] 
Assume that $\xi$ is not spectrally negative and $\xi\in\mathcal{S}_{0}.$ When $q=0$, assume $\e(|\xi_{1}|)<\infty$ and $0<\mu=-\e(\xi_{1})<\infty$, then
$$\p(I>t)\sim \frac{1}{\mu}\int^{\infty}_{\log t}ds\, \overline{\Pi}^{+}(s),\qquad t\to\infty.$$ When $q>0$, we have  $$\p(I>t)\sim \overline{\Pi}^{+}(\log t),\qquad t\to\infty.$$  
\item[(iii)] Assume that  $\xi\in\mathcal{S}_{\alpha}$ for some $\alpha>0,$ and that $\er(e^{\alpha\xi_{1}}, 1<\zeta)<1.$ We have $\e(I^{\alpha})<\infty$ and 
$$\displaystyle \p(I>t)\sim \frac{\e(I^{\alpha})}{-\psi(\alpha)}\overline{\Pi}^{+}(\log t),\qquad t\to\infty,$$  where $\psi(\alpha):=\log\e(e^{\alpha\xi_{1}}1_{\{1<\zeta\}}).$ 
\end{enumerate}
\end{theorem}

The result in ($i$) above has been proved in \cite{rivero1} and in Theorem 2.11 in \cite{patie+savov+1}, it can also be obtained as a corollary to Theorem~\ref{th:asym1} using the estimates for the law of the supremum in \cite{bertoin-doney}. In the case where $q=0$ the result in ($ii$)  was obtained \cite{Maulik-Zwart}, but they had it expressed in terms of the tail distribution of $\xi_{1}$. The version here presented can be obtained as a consequence of Theorem~\ref{th:asym1},  the estimates for the law of the supremum in Theorem 4.1 in \cite{KKM}, and the estimate for the tail L\'evy measure of $H$ in Theorem 3 in \cite{rivero-sinai}. When $q=0$, the result in ($iii$) has also been established in \cite{rivero-convequiv} under the extra assumption that when $0<\alpha<1,$ $\mathbb{E}(|\xi_{1}|)<\infty;$ we can remove this assumption by applying Theorem~\ref{th:asym1} and using the estimates for the law of the supremum in \cite{KKM}. We provide a proof of ($ii$) and ($iii$) in the general case $q\ge0$ in Section \ref{proofs}.

\begin{remark}\label{remarkrenewaltail}
We would like to point out that most of the estimates provided in Theorem~\ref{VRrtb} and Theorem~\ref{corollarysmtg} can also be obtained from our integral equation (\ref{pot-eq}) using renewal theoretic arguments, but a specific, perhaps lengthier, argument would be needed for each of them, we have thus opted for the shortest path. See, however, the proof of Theorem 5.1 in \cite{KRS}, where the argument to prove~($i$) in Theorem~\ref{VRrtb} has been developed in the context of Markov additive processes. This latter argument can be extended to obtain versions of ($i$) in Theorem \ref{VRrtb} when the mean $ \mathbb{E}\left(\xi_1^+ e^{\theta \xi_1}, 1<\zeta \right)$ is infinite using Erickson's strong renewal theorem~\cite{MR0268976}.
\end{remark}

\subsubsection{The behaviour at zero}\label{bezero}
To give a larger panorama of the behaviour of the distribution function of $I$, we will now explain how the results in Section \ref{threeterm} also give the behaviour of the distribution at zero and, in particular, how the three factors in the factorisation~\eqref{equationppsext} of $I$ determine this behaviour. 

In this subsection, we allow the case where $\xi$ is monotone and assume instead that $\xi$ is not spectrally positive. We refer the reader to \cite{patie+savov+book} for details about the case where $\xi$ is spectrally positive.  
  
\begin{theorem}\label{asym:2}
For $\alpha\geq 0,$ the following are equivalent
\begin{enumerate}
\item[(i)] the function $t\mapsto \p(I\leq t)$ is regularly varying at $0$ with index $\alpha,$
\item[(ii)] the function $t\mapsto \p(I_{-\widehat{H}}\leq t)$ is regularly varying at $0$ with index $\alpha.$
\end{enumerate} 
In this case $$\p(I\leq t)\sim \kappa(q,0)\e\left(R^{\alpha-1}_{H}\right)\p\left(I_{-\widehat{H}}\leq t\right),\qquad t\to 0.$$
\end{theorem}

For sake of completeness, we provide below sufficient conditions for the regular variation of the distribution of $I$. From these, we see that the regular variation of the distribution of $I$ holds under very mild assumptions. 

\begin{theorem}\label{corollarysmtg}
Assume that $\xi$ is not spectrally positive and let $\overline{\Pi}^{-}(x):=\Pi(-\infty,-x)$, $x>0$. Then, the following holds.
\begin{enumerate}
\item[(i)] If $q>0,$ then $$\p(I\leq t)\sim q t,\qquad t\to 0+.$$ 
\item[(ii)] If $q=0$, assume that $\overline{\Pi}^{-}$ has exponential decrease, i.e. $\exists \alpha\geq 0$ such that
$$\lim_{x\to \infty}\frac{\overline{\Pi}^{-}(x+y)}{\overline{\Pi}^{-}(x)}=e^{-\alpha y},\qquad y\in \re,$$ and, when $\alpha>0$, the non-Cram\'er condition $\e(e^{-\alpha\xi_{1}})<1$ is satisfied, then $\e(I^{-\alpha})<\infty$ and 
$$\p(I\leq t)\sim \frac{\e(I^{-\alpha})}{1+\alpha}t\,\overline{\Pi}^{-}(\log(1/t))\qquad t\to 0+.$$
\end{enumerate} 
\end{theorem}

\begin{remark}
Similar results to those of Theorem~\ref{corollarysmtg} were obtained in \cite{patie+savov+1}, see  e.g. Theorem 2.15, both for the distribution of $I$ and its density. In particular, a more general version of the item ($i$) above is proved there. We also refer to Corollary 2.9 in the same paper where it is established that the distribution of $I$ may present a polynomial decrease near zero. 
\end{remark}

\subsubsection{Rate of convergence}
\label{rateconvergence}

The identities in Theorems~\ref{PPS} and~\ref{PPS2}  do not lead to second order estimates for the tail distribution of $I$. In this section we discuss how renewal theoretic arguments and the integral equation in Theorem~\ref{keythm} can be used for that end.

\begin{theorem}\label{rateofconvergence}
Assume that $\xi$ has an absolutely continuous (with respect to Lebesgue measure) $1$-resolvent
$$\int^{\infty}_{0}dt\,e^{-t} \p(\xi_{t}\in dy),$$  and that $\xi$ is regular upwards, i.e. 
\begin{equation}
\mathbb{P}\left(T_{(0,\infty)}>0\right)=0,\quad T_{(0,\infty)}:=\inf\{t>0 : \xi_{t}\in(0,\infty)\}.
\end{equation}
If there exist a $\theta>0$ and an integer $m\geq 2$ such that 
\begin{equation}\label{polynomialcond}
\e\left(e^{\theta \xi_{1}}, 1<\zeta\right)=1,\qquad  \e\left(|\xi|^{m+1}_{1}e^{\theta \xi_{1}}, 1<\zeta\right)<\infty,  
\end{equation}
we have
 \begin{align}\label{rateconvergence2}
\sup_{a\leq \delta<\lambda\leq \infty}\left| C_{\theta}t^{\theta}\p(t \delta<I\leq t\lambda)-\int^{\lambda}_{\delta}\frac{dx}{\theta x^{1+\theta}}\right|=o((\log t)^{-(m-1)}),
\end{align}
as $t\to\infty$, where $0<a<\infty$ and $C_{\theta}=\frac{\e\left(\xi_{1}e^{\theta \xi_{1}}, 1<\zeta\right)}{\e(I^{\theta-1})}\in(0,\infty)$. If in addition there is a $\rho>0$ such that
 \begin{equation}\label{exponentialcond}
  \e\left(e^{(\theta+\rho) \xi_{1}}, 1<\zeta\right)<\infty, 
 \end{equation}
then there exists a $0<\gamma<\min\{1, \rho\}$ such that the rate of convergence in~\eqref{rateconvergence2} is of order $o(t^{-\gamma}).$
\end{theorem}

\begin{remark}
Quantifying the rate of convergence of $t^{\theta}\p(I>t)$ towards  $C_{\theta}$ could also be determined using results about the rate of convergence in Kesten's renewal theorem obtained in \cite{goldie} and \cite{BDP}. However, the expression for the error in \cite{goldie} seems rather difficult to make precise and covers only the rate of convergence when there are exponential moments. Besides, the combined moment conditions in \cite{BDP} are less explicit than the conditions provided here, although the structure for the renewal measure required in \cite{BDP} could hold under less regularity on the law of $\xi$ than the one stated in Theorem~\ref{rateofconvergence}.  As in \cite{goldie}, the proof of Theorem~\ref{rateofconvergence} (given in Section~\ref{proofs}) relies in Stone's decomposition theorem~\cite{stone} for a renewal measure with spread-out step distribution, this explains why we require $\xi$ to have an absolutely continuous $1$-resolvent. See also \cite{patie+savov+1} where they circumvent these restrictions by using the Mellin transform of $I,$ and they provide sharper estimates. Theorem~\ref{rateofconvergence} emphasises mainly some of the methods available from renewal theory. 
\end{remark}

In \cite{haas-rivero} it has been shown that under the assumptions in ($i$) of Theorem~\ref{VRrtb}, the exponential functional $I$ belongs to the maximum domain of attraction of a Fr\'echet distribution of parameter $\theta$. In extreme value theory, it is well known that this is equivalent to have that the excess distribution of $I,$ defined via its tail $\overline{F}_{t},$ as
$$\overline{F}_{t}(x):=\frac{\p(I-t>x)}{\p(I>t)},\qquad x \geq 0,$$ satisfies 
$$\overline{F}_{t}(xt)\xrightarrow[t\to\infty]{}\frac{1}{(1+x)^{\theta}},\qquad x\geq 0.$$ 
As a corollary to Theorem~\ref{rateofconvergence} we have the following result, which in our specific case complements what could be derived from known results in the topic, for instance, those in \cite{MR2014267}.

\begin{corollary}\label{corollary...3}
Under the assumptions and notation of Theorem~\ref{rateofconvergence}, when~\eqref{polynomialcond} holds we have the uniform estimate
$$\sup_{x>0}\left|\,\overline{F}_{t}(x)-\frac{1}{(1+x/t)^{\theta}}\right|=o((\log t)^{-(m-1)}),\qquad\text{as}\quad t\to\infty.$$
If in addition \eqref{exponentialcond} holds, the order in the above estimate is polynomial, that is, the right-hand side can be replaced by $o(t^{-\gamma})$.  
\end{corollary}

\section{Further consequences}
\label{theothereq}

We use the notation of Section~\ref{Notation}. Recall the decomposition of the renewal measure $U$ in terms of the renewal measures $V_{H}$ and $V_{\widehat{H}}$ of the upward, resp. downward, ladder height process of $\xi$, given in subsection~\ref{subsection11}. 

Let $\widetilde{\nu}_{H}$ be the measure on $[0,\infty)$ defined by 
\begin{equation}\label{nutilde}
\widetilde{\nu}_{H}(dy)=\left(b_{H}\delta_{0}(dy)+(\overline{\Pi}_{H}(y)+\kappa(q,0))dy\right),\qquad y\geq 0.
\end{equation}
The Laplace transform of $\widetilde{\nu}_{H}$ is
\begin{equation}\label{nutildelaplace}
\frac{\kappa(q,\lambda)}{\lambda}=\int_{[0,\infty)}\widetilde{\nu}_{H}(dx) e^{-\lambda x},
\end{equation} 
while that of $V_{H}$ is 
\begin{equation}
\frac{1}{\kappa(q,\lambda)}=\int_{[0,\infty)}V_{H}(dy)e^{-\lambda y}.
\end{equation}
This implies that the measure obtained by convolution of $\widetilde{\nu}_{H}$ and $V_{H}$ is equal to  the Lebesgue measure on $(0,\infty).$ The same properties hold for the objects defined in terms of $\widehat{H}.$

If $q=0,$ and the mean of $\xi$ is finite, that is $\er\left(|\xi_{1}|\right)<\infty$ and $-\infty<m=\er\left(\xi_{1}\right)<0$, then $-m=\e(I^{-1})$ (see e.g. Theorem 6 in \cite{BY}) and also the downward ladder height process has infinite lifetime, has a finite mean and the following formula holds (see Corollary 4, Chapter 4 in  \cite{doney-book}):
\begin{equation}\label{prod-mean}
-\e(\xi_{1})=\kappa(0,0)\e(\widehat{H}_{1})=\kappa(0,0) \left(\widehat{b}+\int^{\infty}_{0}dx\,\overline{\Pi}_{\widehat{H}}(x)\right).
\end{equation}
In this case, the probability measure $ \nu_{\widehat{H}}=\frac{1}{\e(\widehat{H}_{1})}\widetilde{\nu}_{\widehat{H}}$ is the limiting distribution of $x+\xi_{\tau^{-}_{-x}}$, as $x\to\infty$ (see Exercise 7.9 in \cite{kyprianou-book}), that is,
\begin{equation}\label{undershoot}
\nu_{\widehat{H}}=\lim_{x\to\infty}x+\xi_{\tau^{-}_{-x}},
\end{equation}
where the limit on the right-hand side is in the weak sense. This fact, applied to the equation (\ref{pot-eq}), replacing $t$ by $te^{u}$ and integrating with respect to $\nu_{\widehat{H}}(du),$ directly leads to the following corollary. 

\begin{corollary}\label{cor:2}
Assume that $\xi$ drifts towards $-\infty$ and its mean is finite: $\er\left(|\xi_{1}|\right)<\infty$ and $-\infty<\er\left(\xi_{1}\right)<0$. Let $\mathcal{U}$ be a random variable with distribution~\eqref{undershoot}.  Then, the following identity holds
\begin{equation}\label{furthercons}
e^{-\mathcal{U}}I\stackrel{Law}{=}e^{\overline{\xi}_{\infty}}L,
\end{equation}
where the factors in both sides of the equality are assumed to be independent and $L$ is a random variable with law given by
$$\p\left(L\in dt\right)=\frac{1}{\e(I^{-1})}\frac{1}{t}\p(I\in dt),\qquad t>0.$$
\end{corollary}
Identity~\eqref{furthercons} has been observed in the case where $\xi$ is the negative of a subordinator by Bertoin and Caballero~\cite{bertoin-caballero}. In the particular case where $\xi$ is spectrally positive, that is, it has non-negative jumps, and hence $\mathcal{U}\equiv0,$ the above equality has been obtained by Bertoin and Yor~\cite{BY-moments}. In terms of the densities, \eqref{furthercons} becomes
\begin{equation}\label{eq:5}
\int_{[0,\infty)}\nu_{\widehat{H}}(dz)e^{z}k(xe^{z})=\frac{1}{x \e(\widehat{H}_{1})}\int_{[0,\infty)}V_{H}(dt)k(xe^{-t}).
\end{equation}
If in equation (\ref{eq:5}) we replace $x$ by $xe^{-w}$ and integrate with respect to $\widetilde{\nu}_{H}(dw)$, we get
\begin{equation*}
\begin{split}
\iint_{[0,\infty)^{2}}\widetilde{\nu}_{H}(dw)\nu_{\widehat{H}}(dz)xe^{z-w}k(xe^{z-w})&=\frac{1}{ \e(\widehat{H}_{1})}\int_{[0,\infty)}dt\,k(xe^{-t})\\
&=\frac{1}{\e(\widehat{H}_{1})}\int^{x}_{0}ds\,\frac{1}{s}k(s).
\end{split}
\end{equation*}
Then, using the explicit expression of the measures $\widetilde{v}_{H}$ and $v_{\widehat{H}},$ and the fact that $Q^{2}/2=b\widehat{b}$, we obtain the following corollary which is an extension of an integro-differential equation obtained by Carmona, Petit and Yor \cite{CPY} (see also \cite{BY, bertoin+lindner+maller, Pardo+rivero+Vanschaik, BLR, Kuznetsov-pardo-savov} for related equations).

\begin{corollary}\label{int-equation}
Assume $q=0$ and that $\xi$ has a finite mean: $\e(|\xi_{1}|)<\infty$  and $\e(\xi_{1})<0.$ We have the following identity for $x>0$
\begin{equation}
\begin{split}
&\int^{x}_{0}ds\,\frac{1}{s}k(s)\\
&=\frac{Q^{2}}{2}xk(x)+\widehat{b}\int^{\infty}_{0}dw\left(\overline{\Pi}_{H}(w,\infty)+\kappa(0,0)\right)xe^{-w}k(xe^{-w})\\
&+b\int^{\infty}_{0}dz\,\overline{\Pi}_{\widehat{H}}(z,\infty)xe^{z}k(xe^{z})\\
&+\int^{\infty}_{0}dz\int^{\infty}_{0}dw\left(\overline{\Pi}_{H}(w,\infty)+\kappa(0,0)\right)\overline{\Pi}_{\widehat{H}}(z,\infty)xe^{z-w}k(xe^{z-w}).
\end{split}
\end{equation}
\end{corollary}

\section{Proofs}
\label{proofs}

In this section, we collect the proofs of all the results of the previous sections as presented in the article.

\begin{proof}[Proof of Lemma \ref{lemma:fR}]
Let $\sigma$ be a subordinator and $\widehat{X}$ be the $1$-self-similar Markov process associated to $\widehat{\sigma}=-\sigma$ via Lamperti's transformation, see \cite{lamperti, kyprianou-book}. That is, $\widehat{X}$ is a positive valued strong Markov process killed at its first hitting time of $0,$ $T_{0}:=\inf\{t>0: \widehat{X}_{t}=0\},$ and it has the scaling property: for $c>0,$ the law of the process $(c\widehat{X}_{t/c},\ t\geq 0)$ started  from $\widehat{X}_{0}=x$ equals that of $(\widehat{X}_{t}, t\geq 0)$ started from $\widehat{X}_{0}=cx.$ Furthermore, it can be represented as  
 $$\widehat{X}_{t}=x\exp\{\widehat{\sigma}_{\tau(t/x)}\}1_{\{\tau(t/x)<\infty\}},\qquad \tau(t)=\inf\left\{s>0: \int^{s}_{0}du\exp\{\widehat{\sigma}_{u}\}>t\right\},$$ where $\widehat{X}_{0}=x>0$ and with the usual convention that $\inf \emptyset=\infty.$
We will denote by $\widehat{\p}_{z}$ the law of $\widehat{X}$ started from $\widehat{X}_{0}=z>0.$ 
  In \cite{BY2001} it has been proved that the law of $R_{\sigma}$ is a quasi-stationary law for $\widehat{X}$, that is, the following identity of measures holds for any $t\geq 0$
\begin{equation*}\label{qs}
e^{-t}\p(R_{\sigma}\in dy)=\int_{[0,\infty)}\p(R_{\sigma}\in dz)\widehat{\p}_{z}(\widehat{X}_{t}\in dy, t<T_{0}),\qquad \text{on} \ [0,\infty).
\end{equation*} 
 Integrating the above equation in $t$ we get 
 \begin{equation}\label{qsint}
\frac{1}{y}\p(R_{\sigma}\in dy)=\int_{[0,\infty)}\p(R_{\sigma}\in dz)\int^{\infty}_{0}dt\widehat{\p}_{z}(\widehat{X}_{t}\in dy, t<T_{0})\frac{1}{y},
\end{equation} 
on $[0,\infty).$
The self-similarity and Lamperti's transformation give the following identities, for $f:\re^{+}\to\re^{+}$ positive and measurable, and any $z>0$,
\begin{equation*}
\begin{split}
\int^{\infty}_{0}dt\widehat{\e}_{z}\left(f(\widehat{X}_{t})\frac{1}{\widehat{X}_{t}}, t<T_{0}\right)
&=\int^{\infty}_{0}\frac{dt}{z}\widehat{\e}_{1}\left(f\left(z\widehat{X}_{t/z}\right)\frac{1}{\widehat{X}_{t/z}}, t<T_{0}\right)\\
&=\int^{\infty}_{0}ds\widehat{\e}_{1}\left(f\left(z\widehat{X}_{s}\right)\frac{1}{\widehat{X}_{s}}, t<T_{0}\right)\\
&={\e}\left(\int^{\infty}_{0}ds f\left(z\exp\{\widehat{\sigma}_{\tau(s)}\}\right)\exp\{-\widehat{\sigma}_{\tau(s)}\}1_{\{\tau(s)<\infty\}}\right)\\
&=\e\left(\int^{\infty}_{0}dv f\left(z\exp\{{-\sigma}_{v}\}\right)1_{\{v<\zeta\}}\right)\\
&=\int_{[0,\infty)}V_{\sigma}(dy)f(ze^{-y})
\end{split}
\end{equation*}
 where in the first equality we use the self-similarity, in the second the change of variables $s=tz^{-1},$ in the third we use Lamperti's transformation, and finally in the fourth we make the change of variables $v=\tau(s).$ The above identity and the equality of measures (\ref{qsint}) imply that
\begin{equation*}
\begin{split}
\int_{[0,\infty)}\frac{1}{y}\p(R_{\sigma}\in dy) f(y)&=\int_{[0,\infty)}\p(R_{\sigma}\in dz)\int^{\infty}_{0}dt\widehat{\e}_{z}\left(f(\widehat{X}_{t})\frac{1}{\widehat{X}_{t}}, t<T_{0}\right)\\
&=
\int_{[0,\infty)}\p(R_{\sigma}\in dz)\int_{[0,\infty)}V_{\sigma}(dy)f(ze^{-y}),
\end{split}
\end{equation*} 
for any $f:\re^{+}\to\re^{+}$ positive and measurable. This implies the claimed equality of measures. 
\end{proof}

\begin{proof}[Proof of Theorem \ref{keythm}]

We show that the left and right-hand side of the equation~\eqref{pot-eq} have the same Laplace transform.

On the one hand, we have the pathwise identity
\begin{equation}\label{Markovprop}
1-e^{-\lambda I}=\lambda\int^{\zeta}_{0}dt\,e^{\xi_{t}}\exp\left\{-\lambda\int^{\zeta}_{t}ds\,e^{\xi_{s}}\right\}
=\lambda\int^{\zeta}_{0}dt\,e^{\xi_{t}}\exp\left\{-\lambda e^{\xi_{t}}\widetilde{I}\right\},
\end{equation}
where
$$\widetilde{I}:=\int^{(\zeta-t)}_{0}ds\,e^{\xi_{t+s}-\xi_{t}}.$$ 
On the event $\{t<\zeta\}$, the random variable $\widetilde{I}$ has the same law as $I$ and it is independent of $(\xi_{u}, u\leq t)$, then, taking expectations on both sides of~\eqref{Markovprop}, we have
\begin{equation*}
\begin{split}
&\e\left(1-e^{-\lambda I}\right)\\
&=\lambda\int^{\infty}_{0}dt\e\left(e^{\xi_{t}}\exp\left\{-\lambda e^{\xi_{t}}\widetilde{I}\right\}1_{\{t<\zeta\}}\right),\quad \\
&=\lambda\int^{\infty}_{0}dt\e\left(e^{\xi_{t}}\e\left(\exp\left\{-\lambda e^{\xi_{t}}\widetilde{I}\right\}\big |\mathcal{F}_{t}\cap\{t<\zeta\} \right)\right)\\
&=\lambda\int^{\infty}_{0}dt\e\left(e^{\xi_{t}}\e\left(\exp\left\{-\lambda z\widetilde{I}\right\}\right)\Big|_{z=e^{\xi_{t}}}\right).
\end{split}
\end{equation*}
Using~\eqref{eq:potentialmeasure}, the last term in the expression above can be written as
\begin{equation}\label{1st}
\lambda\int_{\re}U(dy)e^{y}\e\left(\exp\left\{-\lambda e^{y}\widetilde{I}\right\}\right)=\lambda\int_{0}^{\infty}e^{-\lambda u}\int_{\re}U(dy)e^{y}\p(e^{y}I\in du).
\end{equation}

On the other hand, an application of Fubini's theorem implies that for any $\lambda\geq 0,$
\begin{equation}\label{2nd}
\e(1-e^{-\lambda I})=\lambda\int^{\infty}_{0}dt\,e^{-\lambda t}\p(I>t).
\end{equation}
Finally, by~\eqref{1st} and~\eqref{2nd}, we have the equality of measures
$$du\p(I>u)=\int_{\re}U(dy)e^{y}\p(e^{y}I\in du),\qquad \text{on}\,\, (0,\infty).$$
Since $I$ has a density, $k,$ a change of variables implies that 
$$\p(e^{y}I\in du)=e^{-y}k(ue^{-y})du,\qquad \text{on}\,\, (0,\infty).$$ The result follows.
\end{proof}

\begin{proof}[Proof of Corollary \ref{corollary...1}]
It follows by integrating in both sides of the integral equation~\eqref{pot-eq} with respect to $\beta t^{\beta-1}dt$ on $(0,\infty)$, with $\beta\in \mathbb{C}$, $\Re\beta>0$, and $\beta\neq 0$. The proof is indeed elementary but needs the remark that, since we are assuming that the L\'evy process $\xi$ either has a finite lifetime or drifts towards $-\infty$, we have that the measure $U$ is $\sigma$-finite and, for $\lambda\in C,$
$$\int_{\re}U(dy) e^{\lambda y}=\frac{1}{\Psi(-i\lambda)},\qquad \lambda\neq 0.$$ See e.g. the Chapter 1 in \cite{bertoin-book}.  
\end{proof}

\begin{proof} [Proof of Theorem~\ref{PPS2}]

This is a consequence of Lemma~\ref{lemma:fR} and equation~\eqref{supremum}. Indeed, assume that the subordinator $\sigma$ in Lemma~\ref{lemma:fR} has finite lifetime a.s., then $\phi(0)V_{\sigma}(dy)$ defines a probability measure and, if $G^{(0)}$ denotes a random variable with law given by $\phi(0)V_{\sigma}(dy)$ and $J_{\sigma}$ be a random variable with law 
$$\p(J_{\sigma}\in dy)={\phi(0)y}\p\left(\frac{1}{R_{\sigma}}\in dy\right),\qquad y>0,$$ 
then the equality of measure in Lemma~\ref{lemma:fR} can be interpreted as the identity in law
$$J_{\sigma}\stackrel{Law}{=}e^{G^{(0)}}\frac{1}{R_{\sigma}}.$$
To obtain the first identity in Theorem~\ref{PPS2}, we just need to set $\sigma=H$ and use equation~\eqref{supremum}, the second identity then follows from the first one.
\end{proof}

\begin{proof} [Proof of Proposition \ref{corollary....2}]
For the first part, using that the residual exponential functional $R_{H}$ has entire moments of any order, it is enough to justify that, for any $\delta\in(0,1),$ we have $\e\left(I^{-\delta}_{-\widehat{H}}\right)<\infty.$ But this is a straightforward consequence of the factorisation in Lemma~\ref{lemmaBY} and the fact that the exponential random variable has moments of order $-\delta$ for any $\delta\in(0,1)$. 

For the second part, the fact that if $q=0$ and $\e(\xi_{1})<\infty$ then $\e(I^{-1})<\infty$ was proved in Theorem 6 in \cite{BY}. Nevertheless, using the integral equation (\ref{pot-eq}) and the factorisation identity (\ref{fact-pps}), we give below an alternative proof of this fact and also prove the converse. For that end, we establish the result for the downward ladder height subordinator $\widehat{H}$. The general case will follow from the identity (\ref{fact-pps}), because the latter implies that $\e(I^{-1})=\kappa(q,0)\e\left(I^{-1}_{-\widehat{H}}\right),$ and we also have that $\kappa(q,0)>0$ if either $q>0$ or $q=0$ and $\xi$ drifts to $-\infty.$ 

Recall the measure $\widetilde{\nu}_{\widehat{H}}(dy)$ defined in (\ref{nutilde}) and that the convolution of $\widetilde{\nu}_{\widehat{H}}(dy)$ and $V_{\widehat{H}}$ equals the Lebesgue measure. We obtain, by integrating the integral equation (\ref{pot-eq}) with respect to $\widetilde{\nu}_{\widehat{H}}$, that for all $x\in \re,$
\begin{equation*}
\begin{split}
&\e\left(\frac{1}{I_{-\widehat{H}}}1_{\{I_{-\widehat{H}}>e^{x}\}}\right)=\int_{[0,\infty)}\widetilde{\nu}_{\widetilde{H}}(dy)\p(I_{-\widehat{H}}>e^{x+y})\\
&=b\p(I_{-\widehat{H}}>e^{x})+\int^{\infty}_{0}dy\left(\widehat{\kappa}(q,0)+\overline{\Pi}_{\widehat{H}}(y)\right)\p(I_{-\widehat{H}}>e^{x+y}).
\end{split}
\end{equation*}
Then, if $\e(I^{-1}_{-\widehat{H}})<\infty,$ we infer, taking $x\to-\infty$ and using Fatou's theorem, that necessarily $\int^{\infty}_{0}dy\left(\widehat{\kappa}(q,0)+\overline{\Pi}_{\widehat{H}}(y)\right)<\infty$ and hence $\widehat{\kappa}(q,0)=0$ and $\int^{\infty}_{0}dy\overline{\Pi}_{\widehat{H}}(y)<\infty$. The temporal Wiener-Hopf factorisation implies that $q=\kappa(q,0)\widehat{\kappa}(q,0)=0$. It also follows that $\e(\widehat{H}_{1})<\infty.$ Reciprocally, if $q=0$ and $\e(\widehat{H}_{1})<\infty,$ then we have $\widehat{\kappa}(0,0)=0$ and $\int^{\infty}_{0}dy\overline{\Pi}_{\widehat{H}}(y)<\infty.$ Using the monotone convergence theorem, we deduce from the above equation that   
$\e\left(I^{-1}_{-\widehat{H}}\right)=\e(\widehat{H}_{1}).$
\end{proof}

\begin{proof}[Proof of Theorem \ref{th:asym1}]

($ii$) implies ($i$). If ($ii$) is satisfied we have that $\e(e^{\beta \overline{\xi}_{\infty}})<\infty$ for any $\beta<\alpha,$ and by the spatial Wiener-Hopf factorisation, see Proposition 5.1 in \cite{KKM}, this implies that $$0<\e\left(e^{\beta \xi_{1}}1_{\{1<\zeta\}}\right)<1,\qquad \beta\in (0,\alpha).$$ By Lemma~3 in \cite{rivero-convequiv}, the latter implies that $\e(I^{\beta})<\infty$ for any $0<\beta<\alpha.$ The factorisation in Theorem~\ref{PPS2} implies that $\e\left({I}^{\beta}_{-\widehat{H}}\right)<\infty$ and, more importantly, $$\e\left(R^{-\beta-1}_{H}\right)<\infty,\qquad \beta<\alpha.$$ Because of this and the fact that ${I}_{-\widehat{H}}$ has entire moments of any order, we know that there is an $\epsilon>0$ such that $$\e\left(\left({I}_{-\widehat{H}} R^{-1}_{H}\right)^{\alpha+\epsilon}\right)<\infty.$$ By  Breiman's classical result (see Proposition 3 in \cite{breiman}) we obtain 
$$\p\left(e^{\overline{\xi}_{\infty}}{I}_{-\widehat{H}}R^{-1}_{H}>t\right)\sim \e\left(\left({I}_{-\widehat{H}} R^{-1}_{H}\right)^{\alpha}\right)\p\left(e^{\overline{\xi}_{\infty}}>t\right),\qquad t\to\infty.$$ We conclude the proof by recalling the factorisations in Theorem~\ref{PPS2}.

($i$) implies ($ii$). Let us assume that ($i$) holds. In view of the second factorisation in Theorem~\ref{PPS2} and according to Theorem~4.2 in \cite{jacobsenetal}, ($ii$) holds true as soon as $$\e\left(({I}_{-\widehat{H}}R^{-1}_{H})^{\alpha+i\lambda}\right)\neq 0,\qquad \lambda\in \re,$$ whenever $\e(({I}_{-\widehat{H}}R^{-1}_{H})^{\alpha+\delta})<\infty$ for some $\delta>0.$ To verify these facts, we will start by showing that $$\e\left(({I}_{-\widehat{H}}R^{-1}_{H})^{\alpha+\delta}\right)<\infty,$$ for some $\delta>0.$ Since ${I}_{-\widehat{H}}$ has entire moments of any order, it suffices to prove that $\e(R^{-\alpha-\delta}_{H})<\infty.$ From the regular variation of the tail distribution of $I$ it follows, as above, that $\e(R^{-\beta-1}_{H})<\infty,$ for all $\beta<\alpha.$ Next, in the papers~\cite{berg, AJR} it has been proved that the logarithm of the residual exponential functional of any subordinator is infinitely divisible, and so is $\log R_{H}.$ Moreover, the probability measure 
$$\p^{(\alpha)}(R_{H}\in dy):=\frac{y^{-\alpha}}{\e\left(R^{-\alpha}\right)}\p(R_{H}\in dy),$$ is an Esscher transform of the law of $\log R_{H}.$ According to Theorem~33.1 in \cite{sato}, under this probability measure, $\log R_{H}$ is still an infinitely divisible random variable and therefore its characteristic function is nonzero: $$\e^{(\alpha)}\left(\exp\{i\lambda \log R_{H}\}\right)=\e^{(\alpha)}\left(R_{H}^{i\lambda }\right)=\frac{\e(R^{i\lambda-\alpha})}{\e(R^{-\alpha})}\neq 0.$$ We are left to prove that $$\e\left({I}^{\alpha+i\lambda}_{-\widehat{H}}\right)\neq 0,\qquad \lambda\in \re.$$   By the factorisation in Lemma~\ref{lemmaBY} for the exponential random variable  we have the equality 
$$\Gamma(1+\beta+i\lambda)=\int^{\infty}_{0}dx\, x^{\beta+i\lambda}e^{-x}=\e\left({I}^{\beta+i\lambda}_{-\widehat{H}}\right)\e\left(R^{\beta+i\lambda}_{\widehat{H}}\right),\qquad \lambda\in\re,\,\,\beta >-1,$$ where $\Gamma(z)$ denotes the usual Gamma function at $z\in \mathbb{C},$ with $\Re z>0.$ Since the Gamma function has no zeros, our claim follows.
\end{proof}

\begin{proof}[Proof of Theorem~\ref{VRrtb}] 

For $q>0$ the results in ($ii$) and ($iii$) can be obtained adapting the results in Theorem 4.1 in  \cite{KKM}, which were proved under the assumption that $q=0,$ but their arguments allow to guarantee the validity of the results for $q\geq 0$ as follows. First, mimicking the arguments in Proposition 17 in Section VI in~\cite{bertoin-book}, it is not difficult to verify that the law of the supremum of $\xi$ on $(0,\zeta)$ is equal to the law of a subordinator with infinite lifetime $\mathcal{H}$ and Laplace exponent $\lambda\mapsto\kappa(q,\lambda)-\kappa(q,0),$ sampled at an independent exponential time of parameter $\kappa(q,0),$ that we will denote by $\ee_{\kappa(q,0)},$ viz.
$$\overline{\xi}_{\infty}=\sup_{0<s<\zeta}\xi_{s}\stackrel{Law}{=}\mathcal{H}_{\ee_{\kappa(q,0)}}.$$ Moreover, the upward ladder height $H$ has the same law as $\mathcal{H}$ killed at $\ee_{\kappa(q,0)}$. According to Exercise 7.5 in \cite{kyprianou-book}, the tail L\'evy measure of $\mathcal{H}$ is obtained from $\Pi$ by Vigon's identity and
$$\overline{\Pi}_{H}(x)=\overline{\Pi}_{\mathcal{H}}(x)={\Pi}_{\mathcal{H}}(x,\infty)=\int_{[0,\infty)}V_{\widehat{H}}(dy)\overline{\Pi}(x+y,\infty),\qquad \qquad x>0.$$ Since $\widehat{H}$ is also a killed subordinator, its potential measure $V_{\widehat{H}}(dy)$ is a renewal measure with finite total mass and Laplace transform
$$\int_{[0,\infty)}V_{\widehat{H}}(dy)e^{-\lambda y}=\frac{1}{\widehat{\kappa}(q,\lambda)},\qquad \lambda \geq 0.$$
The arguments in the proof of Proposition 5.3 in \cite{KKM} apply to show that when $\alpha>0,$ if $\overline{\Pi}^{+}$ has an exponential decrease with index $\alpha,$ then $\overline{\Pi}_{\widehat{H}}$ holds the same property and also
$$\overline{\Pi}^{+}(x)\sim \widehat{\kappa}(q,\alpha)\overline{\Pi}_{H}(x),\qquad x\to\infty.$$ We can now follow the proof of Theorem 4.1 in \cite{KKM}, both when $\alpha>0$ and $\alpha=0,$ to deduce that
$$\p\left(\mathcal{H}_{\ee_{\kappa(q,0)}}>x\right)\sim \frac{\kappa(q,0)}{(\kappa(q,-\alpha))^{2}}\overline{\Pi}_{\mathcal{H}}(x),\qquad x\to\infty.$$ 
When $\alpha>0,$ the rightmost term above behaves asymptotically, as $x\to\infty$, like $$\frac{\kappa(q,0)}{\kappa(q,-\alpha)}\frac{1}{\kappa(q,-\alpha)\widehat{\kappa}(q,\alpha)}\overline{\Pi}^{+}(x).$$
In this case, the Wiener-Hopf factorisation implies that $\kappa(q,-\alpha)\widehat{\kappa}(q,\alpha)=-\psi(\alpha),$ with $\psi(\alpha)$ as defined in the statement of the theorem. According to Theorem~\ref{th:asym1}, we have the estimate
\begin{equation*}
\begin{split}
\p(I>t)&\sim \e\left(I^{\alpha}_{-\widehat{H}}R^{-\alpha}_{H}\right)\frac{\kappa(q,0)}{(\kappa(q,-\alpha))^{2}}\overline{\Pi}_{H}(\log t)\\
&\sim \e\left(I^{\alpha}_{-\widehat{H}}R^{-\alpha}_{H}\right)\frac{\kappa(q,0)}{\kappa(q,-\alpha)}\frac{1}{-\psi(\alpha)}\overline{\Pi}^{+}(\log t).
\end{split}
\end{equation*}
 To finish the case $\alpha>0$ we just need to verify the equality of constants
$$\e(I^{\alpha})=\e\left(I^{\alpha}_{-\widehat{H}}R^{-\alpha}_{H}\right)\frac{\kappa(q,0)}{\kappa(q,-\alpha)}.$$ But this follows from the identity in law in Theorem~\ref{PPS} and Lemma~\ref{lemma:fR}. Now, when $\alpha=0,$ Theorem~\ref{th:asym1} ensures that
$$\p(I>t)\sim \frac{1}{\kappa(q,0)}\overline{\Pi}_{H}(\log t),\qquad t\to\infty.$$ Therefore, we are left to verify that
$$\frac{1}{\kappa(q,0)}\overline{\Pi}_{{H}}(x)\sim \overline{\Pi}^{+}(x),\qquad x\to\infty.$$ This is an immediate consequence of the dominated convergence theorem (as in the proof of Theorem 3 (b) in \cite{rivero-sinai}).
\end{proof}

\begin{proof}[Proof of Theorem \ref{asym:2}]
Analogously to the proof of Theorem~\ref{th:asym1}, the equivalence is a consequence of Theorem~4.2 in \cite{jacobsenetal}. Indeed, because of the identity in Theorem~\ref{PPS}, we have
$$\p(I^{-1}>t)=\p\left(({I}_{-\widehat{H}}J_{H})^{-1}>t\right),\qquad t>0.$$ 
Then, it is enough to verify that 
$$\e\left((J^{-1}_{H})^{\alpha+i\lambda}\right)=\kappa(q,0)\e\left(R^{\alpha-1+i\lambda}_{H}\right)\neq 0,\qquad \lambda \in \re,$$ whenever $\e((J^{-1}_{H})^{\alpha+\delta})<\infty$ for some $\delta>0.$ The former is immediate from the infinite divisibility of $\log R_{H}$. The latter follows from the fact that $R_{H}$ has entire moments of any order. Therefore, Theorem~4.2 in \cite{jacobsenetal} ensures that ($i$) and ($ii$) are equivalent. Moreover, since for any $\epsilon>0$ we have that 
$$\e\left((J^{-1}_{H})^{\alpha+\epsilon}\right)={\kappa(q,0)}\e\left(R^{\alpha+\epsilon-1}_{H}\right)<\infty,$$ 
Breiman's result (see Proposition 3 in \cite{breiman}) then implies the estimate $$\p(I\leq t)\sim {\kappa(q,0)}\e\left(R^{\alpha-1}_{H}\right)\p(I_{-\widehat{H}}\leq t),\qquad t\to 0+.$$ 
\end{proof}

\begin{proof}[Proof of Theorem \ref{corollarysmtg}]
The proof of this Theorem is based on the results of \cite{Pardo+rivero+Vanschaik}. When $q>0,$  both the upward and downward ladder height processes are killed subordinators. Then, Theorem~2.5 in \cite{Pardo+rivero+Vanschaik} implies that 
$$\p({I}_{-\widehat{H}}\leq t)\sim \widehat{\kappa}(q,0)t,\qquad t\to 0+.$$ By Theorem~\ref{asym:2} we have $$\p(I\leq t)\sim\kappa(q,0)\widehat{\kappa}(q,0)t,\qquad t\to 0+.$$ The temporal Wiener-Hopf factorisation $q=\kappa(q,0)\widehat{\kappa}(q,0)$ then gives the characterisation of the constant, which concludes the proof in this case. 

When $q=0,$ the process $\xi$ drifts to $-\infty$ and the results in Theorem 3 (b) in \cite{rivero-sinai}  (when $\alpha=0$) and Proposition 5.3 in \cite{KKM}  (when $\alpha>0$) allow us to ensure that the tail L\'evy measure of the downward ladder height subordinator $\overline{\Pi}_{\widehat{H}}$ has exponential decrease and 
$$\overline{\Pi}_{\widehat{H}}(x)\sim \frac{1}{\kappa(0,\alpha)}\overline{\Pi}^{-}(x), \qquad x\to \infty.$$  Moreover, Theorem~2.5 in \cite{Pardo+rivero+Vanschaik} implies that 
$$\p({I}_{-\widehat{H}}\leq t)\sim\frac{\e({I}_{-\widehat{H}}^{-\alpha})}{1+\alpha}t\overline{\Pi}_{\widehat{H}}(\log(1/t)),\qquad t\to 0+,$$ and $\e\left({I}^{-\alpha}_{-\widehat{H}}\right)<\infty.$ Putting the pieces together we get
$$\p(I\leq t)\sim\frac{1}{1+\alpha}\e\left({I}_{-\widehat{H}}^{-\alpha}\right)\kappa(q,0)\e\left(R^{\alpha}_{H}\right)\frac{1}{\kappa(0,\alpha)}t\overline{\Pi}^{-}(\log(1/t)),\qquad t\to 0+.$$ The recurrence for the moments of $R_{H}$ gives the identity $\e\left(R^{\alpha}_{H}\right)\frac{1}{\kappa(0,\alpha)}=\e\left(R^{\alpha-1}_{H}\right),$ which together with  the factorisation in (\ref{PPS}), allows to deduce that the constant in the above estimate satisfies the identity
$$\kappa(q,0)\e\left({I}_{-\widehat{H}}^{-\alpha}\right)\e\left(R^{\alpha}_{H}\right)\frac{1}{\kappa(0,\alpha)}=\e(I^{-\alpha}).$$ The above also ensures that $\e(I^{-\alpha})<\infty.$ 
\end{proof}

\begin{proof}[Proof of Theorem \ref{rateofconvergence}]
The main idea is to use Lemma~\ref{febrero} for the L\'evy process whose law is obtained via the Esscher transformation $\p^{\theta}|_{\mathcal{F}_{t}}=e^{\theta\xi_{t}}\p|_{\mathcal{F}_{t}},$ for $t\geq 0.$ For that end, let us verify that, under $\p^{\theta}$, the conditions of Lemma~\ref{febrero} are satisfied. Under $\p^{\theta}$ we have that $$\e^{\theta}(\xi^{m+1}_{1})=\e(\xi^{m+1}_{1}e^{\theta\xi_{1}})<\infty.$$ 
According to Theorem 6.2.3 in \cite{thesevigon} the latter implies that, under $\p^{\theta}$, the upward ladder height has finite moment of order $m$, $\e^{\theta}(H^{m}_{t})<\infty,$ for all $t\geq 0.$ Moreover, the absolute continuity of the $1$-resolvent of $\xi$ under $\p$ with respect to Lebesgue measure, together with the local absolute continuity relation between $\p^{\theta}$ and $\p,$ implies that under $\p^{\theta}$ the $1$-resolvent of $\xi$ is also absolutely continuous. The assumption of $\xi$ being regular upwards under $\p$ implies that the same property holds under $\p^{\theta}$, since this is a local property of $\xi$ and the measures $\p^{\theta}$ and $\p$ are locally equivalent. The regularity upwards implies that the potential measure of the upward ladder height process $H$ under $\p^{\theta}$ has no mass at zero, $V^{\theta}_{H}(\{0\})= 0.$ Theorem 3 in \cite{MR3098676} implies that the law of $H_{\ee_{1}}$ under $\p^{\theta},$ with $\ee_{1}$ being an independent standard exponential random variable, has a density with respect to Lebesgue measure. 

Now, the integral equation in Theorem~\ref{keythm} allows to rewrite the difference
$$\left|t^{\theta}\p(t\delta<I\leq t\lambda)-C_{\theta}^{-1}\int_{\delta}^{\lambda}\frac{\theta dx}{x^{1+\theta}}\right|,$$
in terms of the renewal measure of $\xi$ under $\p^{\theta},$ 
$U_{\theta}(dy)=e^{\theta y}U(dy),$ and the integrable function $k_{\theta}(v):=v^{\theta}(k(\delta v)-k(\lambda v))$, as follows
\begin{align}\label{enero}
\left|\int_{\re}U(dy)t^{\theta}(k(t\delta e^{-y})-k(t\lambda e^{-y}))-\frac{1}{\mu_{\theta}}\int_{(0,\infty)}\frac{dv}{v}v^{\theta}(k(\delta v)-k(\lambda v))\right|\notag\\
=\left|\int_{\re}U_{\theta}(dy)k_{\theta}(te^{-y})-\frac{1}{\mu_{\theta}}\int_{\re}dyk_{\theta}(te^{-y})\right|,
\end{align}
where $\mu_{\theta}=\e^{\theta}(\xi_{1}).$
By hypotheses, under the transformed probability measure $\p^{\theta}$, $\xi$ is a L\'evy process that drifts to $+\infty$, satisfies $0<\e^{\theta}(\xi_{1})<\infty$, and it has renewal measure $U_{\theta}$. Using Lemma~\ref{febrero}, there is a bounded function $q$ such that (\ref{enero}) takes the form
\begin{align}\label{febrero++}
\left|\int_{-\infty}^{\infty}\frac{dy}{\e^{\theta}(H_{1})}k_{\theta}(te^{-y})\left(\int_{0}^{\infty}V_{\widehat{H}}^{\theta}(dz)q(y+z)\right)\right|, 
\end{align}
where $V_{\widehat{H}}^{\theta}$ is the potential measure of the downward ladder height process $\widehat{H}$ under $\p^{\theta}.$
Now, let $$\ell(y):=\int_{0}^{\infty}V_{\widehat{H}}^{\theta}(dz)q(y+z),\qquad y\in \re,$$
and, making the change of variables $u=te^{-y}$ in the first term of (\ref{febrero++}), we get
\begin{align*}
\int_{-\infty}^{\infty}\frac{dy}{\e^{\theta}(H_{1})}k_{\theta}(te^{-y})\ell(y)=\frac{1}{\e^{\theta}(H_{1})}\int_{0}^{\infty}\frac{du}{u}k_{\theta}(u)\ell(\log(t/u)).
\end{align*} Notice that $\ell$ is a bounded function because $q$ is bounded and $V_{\widehat{H}}^{\theta}$ is a finite measure since, under $\p^{\theta}$, $\xi$ drift towards $\infty.$
To deal with the above integral, let $\chi\in(0,1)$ fixed, and split the interval $(0,\infty)$ into the sets $(0,t^{\chi})$ and $[{t^{\chi}},\infty)$. The integral over $[{t^{\chi}},\infty)$ can be bounded by above by
\begin{align*}
\sup_{y<(1-\chi)\log t}|\ell(y)|\int_{{t^{\chi}}}^{\infty}\frac{du}{u}k_{\theta}(u),
\end{align*}
and, using integration by parts and Karamata's theorem, we get
\begin{align*}
\int_{{t^{\chi}}}^{\infty}\frac{du}{u}k_{\theta}(u)&=\frac{1}{\delta^{\theta}}\int_{\delta{t^{\chi}}}^{\infty}dvv^{\theta-1}k(v)-\frac{1}{\lambda^{\theta}}\int_{\lambda{t^{\chi}}}^{\infty}dvv^{\theta-1}k(v)\\
&=\frac{1}{\delta^{\theta}}(\delta{t^{\chi}})^{\theta-1}\p(I>\delta{t^{\chi}})+\frac{\theta-1}{\delta^{\theta}}\int_{\delta{t^{\chi}}}^{\infty}duu^{\theta-2}\p(I>u)\\
&-\frac{1}{\lambda^{\theta}}(\lambda{t^{\chi}})^{\theta-1}\p(I>\lambda{t^{\chi}})-\frac{\theta-1}{\lambda^{\theta}}\int_{\lambda{t^{\chi}}}^{\infty}duu^{\theta-2}\p(I>u)\\
&\sim \left(\frac{1}{\delta^{\theta+1}}-\frac{1}{\lambda^{\theta+1}}\right)t^{-\chi}t^{\theta\chi}\p(I>t^{\chi}).
\end{align*}

The above estimate is uniform both in $\delta$ and $\lambda$, on each interval of the form $[a,\infty)$, $0<a<\infty$, this is because the functions $t\to t^{\theta-1}\p(I>t)$ and $\int_{t}^{\infty}duu^{\theta-2}\p(I>u)$, $t>0$, are regularly varying at infinity with index $-1$, and hence Theorem 1.5.2 in \cite{BGT} applies. Therefore
\begin{align}\label{marzo}
&\sup_{y<(1-\chi)\log t}|\ell(y)|\int_{{t^{\chi}}}^{\infty}\frac{du}{u}k_{\theta}(u)\\
&\sim\theta\sup_{y<(1-\chi)\log t}|\ell(y)|\left(\frac{1}{\delta^{\theta+1}}-\frac{1}{\lambda^{\theta+1}}\right)\frac{1}{t^{\chi}}t^{\theta\chi}\p(I>t^{\chi}),
\end{align}
which has order $t^{-\chi}$ uniformly in $a\leq  \delta\leq \lambda\leq \infty.$ 

Consider now the term of the integral over $(0,{t^{\chi}})$. Since $\ell(y)=o(y^{-(m-1)})$ as $y\to\infty$, and, for $0<u<{t^{\chi}}$ we have $\log(t/u)>(1-\chi)\log t$, taking $\epsilon>0$ and $t$ large enough we obtain
$$\left|\ell(\log(t/u))\right|\leq\epsilon(\log(t/u))^{-(m-1)}.$$
As a result, the integral over $(0,{t^{\chi}})$ can be bounded by above by
$$\int_{0}^{{t^{\chi}}}\frac{du}{u}k_{\theta}(u)(\log(t/u))^{-(m-1)}\leq \epsilon(1-\chi)^{-(m-1)}(\log t)^{-(m-1)}\int_{0}^{{t^{\chi}}}\frac{du}{u}k_{\theta}(u).$$ Since we also have the upper bound 
$$0\leq \int_{0}^{{t^{\chi}}}\frac{du}{u}k_{\theta}(u)\leq \left(\frac{1}{\delta^{\theta}}-\frac{1}{\lambda^{\theta}}\right)\e(I^{\theta-1})\leq 2a^{-\theta}\e(I^{\theta-1}),$$ we conclude that
$$\frac{1}{\e^{\theta}(H_{1})}\int_{0}^{t^{\chi}}\frac{du}{u}k_{\theta}(u)\ell(\log(t/u))=o((\log t)^{-(m-1)}),$$
uniformly in $0<a\leq \delta\leq \lambda\leq \infty.$

If we have furthermore
$$\e(e^{(\theta+\rho)\xi_{1}})<\infty,\quad\text{for some}\quad\rho>0,$$ the classical Stone's decomposition theorem~\cite{stone} ensures that $\ell(y)=o(e^{-\gamma y})$ as $y\to\infty$, for some $0<\gamma<\rho$. We can assume $\gamma<\min\{1, \rho\}$ so that we also have $\e(I^{\theta+\gamma-1})<\infty$.

Going back to the estimate of the integral on $(0,{t^{\chi}}),$ we see that, in this case, for $\epsilon>0$ and $t$ large enough, the corresponding integral is bounded by above by
$$\epsilon\int_{0}^{{t^{\chi}}}\frac{du}{u}k_{\theta}(u)\left(t/u\right)^{-\gamma}.$$ The finiteness of $\e(I^{\theta+\gamma-1})$ then allows us to write
\begin{align}\label{mayo}
\int_{0}^{{t^{\chi}}}\frac{du}{u}k_{\theta}(u)\left(t/u\right)^{-\gamma}&\sim t^{-\gamma}\int_{0}^{\infty}\frac{du}{u^{1-\gamma}}k_{\theta}(u)\notag\\
&=t^{-\gamma}\left(\frac{1}{\delta^{\theta+\gamma}}-\frac{1}{\lambda^{\theta+\gamma}}\right)\e(I^{\theta+\gamma-1}).
\end{align} 
Using (\ref{marzo}) and (\ref{mayo}), we finally obtain an order $o(t^{-\gamma^{\prime}})$, with $\gamma^{\prime}=\min\{\chi, \gamma\}$, which gives the claimed result.
\end{proof}

\begin{proof}[Proof of Corollary \ref{corollary...3}]
The statements of the corollary follow from Theorem~\ref{rateofconvergence} using the inequality
\begin{align*}
\sup_{x>0}\left|\overline{F}_{t}(x)-\frac{1}{(1+x/t)^{\theta}}\right|
&=\sup_{x>0}\left|\overline{F}_{t}(xt)-\frac{1}{(1+x)^{\theta}}\right|\\
&\leq \frac{1}{C_{\theta}t^{\theta}\p(I>t)}\sup_{x>0}\left|C_{\theta}t^{\theta}\p(I>t(1+x))-\frac{1}{(1+x)^{\theta}}\right|\\
&+ \frac{1}{C_{\theta}t^{\theta}\p(I>t)}\left|1-C_{\theta}t^{\theta}\p(I>t)\right|.
\end{align*}
\end{proof}

\appendix

\section{Stone's decomposition theorem}
\label{App:Stone}

The next lemma is a version of Stone's decomposition theorem. We follow the notation of Section~\ref{Notation}.

\begin{lemma}\label{febrero}
Assume $\xi=(\xi_{t}, t\geq 0)$ is real valued L\'evy process with infinite lifetime, finite mean $\e(|\xi_{1}|)<\infty,$ and such that $0<\mu=\e(\xi_{1})<\infty.$ Assume furthermore that the $1$-resolvent of the upward ladder height process $H$, 
$$\int^{\infty}_{0}dte^{-t}\p(H_{t}\in dy),$$ has a density w.r.t. Lebesgue measure on $(0,\infty)$ and that there is an $m\geq 2$ such that $\e(H^{m}_{1})<\infty.$ In this case, there is a bounded function $q$ such that 
$$\frac{dy}{\mu}-U(dy)=\frac{dy}{\e(H_{1})}\left(\int^{\infty}_{0}V_{\widehat{H}}(dz)q(y+z)\right)-V_{\widehat{H}}(-dy)V_{H}\{0\}1_{\{y<0\}},\qquad y\in\re,$$ and 
$$|q(y)|=o(y^{-(m-1)}),\qquad y\to\infty.$$
\end{lemma}

\begin{proof}
The hypothesis of $H$ having a $1$-resolvent that is absolutely continuous with respect to Lebesgue measure is equivalent to the absolute continuity of the law of $H_{\ee_{1}},$ with $\ee_{1}$ a standard exponential random variable independent of $H$. This implies, in particular, that $H_{\ee_{1}}$ has a spread-out distribution. The identity (\ref{renewalU}) in Section~\ref{Notation} implies that the renewal measure of $H$ has a density w.r.t. Lebesgue measure on $(0,\infty)$ that we will denote by $v_{H}$. By hypothesis $\e(H^{m}_{1})<\infty$, then $\e(H^{m}_{t})<\infty$ for all $t\geq 0$ and, since the mapping $t\mapsto \e(H^{m}_{t})$ is a polynomial of order $m$, it follows that $\e\left(H^{m}_{\ee_{1}}\right)<\infty.$  According to the classical Stone's decomposition theorem~\cite{stone}, the density $v_{H}$ is continuous and bounded and can be written as 
$$v_{H}(y)=\frac{1}{\e(H_{1})}+p(y),\qquad y>0,$$ and $p(y)=o(y^{-(m-1)})$ as $y\to\infty.$ Using this, the  Wiener-Hopf factorisation for the potential (\ref{WHpot}) takes the form
$$U(dy)=V_{H}\{0\}V_{\widehat{H}}(-dy)1_{\{y<0\}}+dy\int^{\infty}_{0}V_{\widehat{H}}(dz)v_{H}(y+z)1_{\{y+z> 0\}},\qquad y\in\re.$$ %
Moreover, using the equality in (\ref{prod-mean}) and the fact that $\int^{\infty}_{0}V_{\widehat{H}}(dz)=\frac{1}{\widehat{\kappa}(q,0)},$ with $\widehat{\kappa}(q,0)$ being the killing rate of the downward ladder height process, we obtain the equality
$$\frac{dy}{\mu}=\frac{dy}{\e(H_{1})}\int^{\infty}_{0}V_{\widehat{H}}(dz),\qquad y\in \re.$$
Therefore, for any compactly supported and measurable function $f$ we have the equalities
\begin{equation*}
\begin{split}
&\int_{\re}\frac{dy}{\mu}f(y)-\int_{\re}U(dy)f(y)\\
&=\int^{\infty}_{0}V_{\widehat{H}}(dz)\left(\int_{\re}dyf(y)\left(\frac{1}{\e(H_{1})}-v_{H}(y+z)1_{\{y+z> 0\}}\right)\right)\\
&\quad-\int^{\infty}_{0}V_{\widehat{H}}(dz)V_{H}\{0\}f(-z) \\
&=\int_{\re}dy f(y)\int^{\infty}_{0}V_{\widehat{H}}(dz)\left(\frac{1}{\e(H_{1})}1_{\{y+z<0\}}-p(y+z)1_{\{y+z> 0\}}\right)\\
&\qquad-\int^{\infty}_{0}V_{\widehat{H}}(dz)V_{H}\{0\}f(-z).
\end{split}
\end{equation*}The result follows by taking 
$$q(z)=\frac{1}{\e(H_{1})}1_{\{z<0\}}-p(z)1_{\{z> 0\}},\qquad z\in\re.$$
\end{proof}

\end{document}